\RequirePackage{fix-cm}
\documentclass[smallextended]{svjour3}       
\smartqed  
\usepackage{graphicx}
\usepackage{lineno}
\usepackage{amsmath}
\usepackage{dsfont}
\usepackage{amssymb}
\usepackage[all]{xy}
\usepackage{bbding}
\usepackage{ntheorem}
\usepackage{mathptmx}      
\theoremstyle{plain}\theoremheaderfont{\normalfont\bfseries}
\theoremnumbering{Alph}
\newtheorem{theorem1}{$\mathbf{Theorem}$}

\theoremstyle{nonumberplain}\theoremheaderfont{\bfseries}
\theorembodyfont{\normalfont}
\newtheorem{proof1}{Proof of the Theorem \ref{theorem new1}}

\theoremstyle{nonumberplain}\theoremheaderfont{\bfseries}
\theorembodyfont{\normalfont}
\newtheorem{proof2}{Proof of the Theorem \ref{Corollary new1}}

\AtBeginDocument{%
  \paperwidth=\dimexpr
    1in + \oddsidemargin
    + \textwidth
    + 1in + \oddsidemargin
  \relax
  \paperheight=\dimexpr
    1in + \topmargin
    + \headheight + \headsep
    + \textheight
    + 1in + \topmargin
  \relax
  \usepackage[pass]{geometry}\relax
}
\begin{document}

\title{The Moduli Space of Points in the Boundary of Quaternionic Hyperbolic Space
}
\subtitle{}


\author{ Gaoshun Gou \and
        Yueping Jiang 
}


\institute{Gaoshun Gou \at
              Department of Mathematics, Hunan-University, Changsha {\rm 410082}, P. R. China \\
              \email{gaoshungou@hnu.edu.cn} 
           \and
          \Envelope\, Yueping Jiang \at
              Department of Mathematics, Hunan-University, Changsha {\rm 410082}, P. R. China
              \\ \email{ypjiang@hnu.edu.cn}
 }

\date{}

\maketitle

\begin{abstract}
Let $\mathcal{F}_1(n,m)$ be the space of ordered m-tuples of pairwise distinct points in $\partial \mathbf{H}_{\mathbb{H}}^n$ up to its isometry group $PSp(n,1)$. It is a real $2m^2-6m+5-\sum^{m-n-1}_{i=1}{m-2 \choose n-1+i}$ dimensional algebraic variety when $m>n+1$. In this paper, we construct and describe the moduli space of $\mathcal{F}_1(n,m)$, in terms of the Cartan's angle and cross-ratio invariants, by applying the Moore's determinant.
\keywords{Quaternionic hyperbolic space\and Moduli space\and Moore's determinant\and special-Gram quaternionic matrix\and Invariants}
\subclass{32M15\and 22E40\and 15B33\and 14J10}
\end{abstract}
\section{Introduction}
Moduli spaces are spaces of solutions of geometric classification problems. They can be viewed as giving a universal space of parameters for the problem.
In this paper, we will study the Moduli space of $\mathcal{F}_1(n,m)$ and the Moduli space of the discrete, faithful, totally loxodromic representation family of groups into $PSp(2,1)$.
In projective space, those Moduli spaces are studied by many researchers.

In $\mathbb{RP}^{n}$, Hakim-Sanler used Bruhat decomposition to study $n+1$ points in a certain standard position in \cite{HakimJ.2000}. They obtained some geometry applications of the structure of double coset space arising in representation theory.

In $\mathbb{CP}^{n}$,
the Moduli spaces for the $PU(n,1)$-congruence classes of ordered distinct points have been well studied. There are three kinds of points: negative, isotropic and positive. The case of negative points was solved by Brehm-Et-Taoui in \cite{Brehm2,Brehm3}. The moduli space of ordered quadruples of distinct isotropic points in $\mathbb{CP}^{2}$ was studied by Falbel, Parker, Platis in \cite{Parker2008,Falbel8} etc. Parker and Platis \cite{Parker2008} defined analogous Fenchel-Nielsen coordinates for complex hyperbolic quasi-conformal representations of surface groups. The traces and cross-ratio invariants were used to parameterize the deformation space of complex hyperbolic totally loxodromic quasi-Fuchsian groups. They obtained some important results related to study of totally loxodromic groups. Falbel and Platis \cite{Falbel8} showed the quotient configuration space $\mathcal{F}$ is a real four dimensional variety and proved the existence of natural complex and CR structures on this space. The case of ordered quadruples of distinct isotropic points in $\mathbb{CP}^{n}$ was investigated by Cunha-Gusevskii in \cite{Cunha4}, then they \cite{Cunha5} further studied the moduli space of ordered m-tuple ($m\geq4$) of distinct isotropic points in $\mathbb{CP}^{n}$. The condition $m=n+1$ was discussed by Hakim-Sanler in \cite{HakimJ.16}. Moreover, the positive points case was concerned in \cite{Cunha6}.

In $\mathbb{HP}^{n}$, the development of the moduli space for the space $\mathcal{F}_1(n,m)$ of ordered distinct points is still in its childhood. The triple of distinct negative points in $\mathbf{H}_{\mathbb{H}}^n$ was considered by Brehm in \cite{Brehm2}.
Recently, Cao \cite{Cao} studied the moduli space of the triple and the quadruple of isotropic points in $\mathbb{HP}^{n}$ without associating quaternionic determinant.

It is natural to ask whether the other cases are still true for the moduli space of ordered distinct points in $\mathbb{HP}^{n}$.
Therefore, we restrict our attention in this article to the Moduli space of ordered m-tuple ($m\geq4$) of distinct points in $\partial \mathbf{H}_{\mathbb{H}}^n$.

The Gram matrices and invariants associated with a well defined determinant were used to study the moduli space of points
in $\mathbb{CP}^{n}$. Thus, if we want to gain additional insight into those of $\mathbb{HP}^{n}$, especially to study the relationship between cross-ratio invariants and the Moduli spaces, a suitable quaternionic determinant is essential. In order to bridge this gap, we use the Moore's determinant (see \cite{E.H.Moore7}).

In this paper, the main purpose is to construct and describe the Moduli space of $PSp(n,1)$-congruence classes of an ordered m-tuple ($m\geq4$) of distinct points in $\partial \mathbf{H}_{\mathbb{H}}^n$ by the special-Gram quaternionic matrices and invariants associated with Moore's determinant. For the purpose, we carry it out in two steps.

Firstly, we overcome the abstractness of the space $\mathcal{F}_1(n,m)$ of m-tuple of distinct points in $\partial \mathbf{H}_{\mathbb{H}}^n$. In this place, the special-Gram quaternionic matrices, defined in Definition \ref{Dfinition2}, plays an important role.
Let $[\mathfrak{p}]$ be an element of the space $\mathcal{F}_1(n,m)$ of m-tuple $\mathfrak{p}=(p_1,\ldots,p_m)$ distinct points in $\partial \mathbf{H}_{\mathbb{H}}^n$. We view $[\mathfrak{p}]$ as the ordered m-tuple $\mathfrak{p}$ up to the action of $PSp(n,1)$. In Definition \ref{definition 3}, we give the definition of equivalent relation of special-Gram qutaternionic matrices. Next, we show that a m-tuple $\mathfrak{p}$ corresponding to a normalized special-Gram quaternionic matrix in section 3.1. Then in Proposition \ref{pro6} and Corollary \ref{Corollary1}, we will see that $[\mathfrak{p}]$ is uniquely corresponding to a normalized special-Gram quaternionic matrix. In fact, the Definition \ref{Dfinition2} help us to realize this uniqueness. So then, we can directly study the normalized special-Gram quaternionic matrices instead of the space $\mathcal{F}_1(n,m)$.

Secondly, we parameterize the normalized special-Gram quaternionic matrices. Here, the quaternionic cross-ratios and a quaternionic Cartan angle as the invariants are necessary.
Let $\mathfrak{p}=(p_1,p_2,p_3,p_4)$ be a quadruple of distinct isotropic points in $\partial\mathbf{H}_{\mathbb{H}}^n$. The cross-ratio is defined by $$\mathbb{X}(p_1,p_2,p_3,p_4)=\langle \mathbf{p}_3,\mathbf{p}_1\rangle_1\langle\mathbf{p}_3,\mathbf{p}_2\rangle_1^{-1}\langle\mathbf{p}_4,\mathbf{p}_2
\rangle_1\langle\mathbf{p}_4,\mathbf{p}_1\rangle_1^{-1},$$
where $\mathbf{p}_i$ is the lift of $p_i$: $\mathbf{p}_i\in \mathbb{H}^{n,1}$.
The quaternionic Cartan's angular invariant of triple $(p_1,p_2,p_3)$ is given by:
$$\mathbb{A_H}=\mathbb{A_H}(p_1,p_2,p_3)=
\arccos\frac{\Re(-\langle\mathbf{p}_1,\mathbf{p}_2,\mathbf{p}_3\rangle_1)}
{|\langle\mathbf{p}_1,\mathbf{p}_2,\mathbf{p}_3\rangle_1|},$$ $$\langle\mathbf{p}_1,\mathbf{p}_2,\mathbf{p}_3\rangle_1=\langle\mathbf{p}_1,\mathbf{p}_2\rangle_1
\langle\mathbf{p}_2,\mathbf{p}_3\rangle_1\langle\mathbf{p}_3,\mathbf{p}_1\rangle_1.$$ Where $$\Re(-\langle\mathbf{p}_1,\mathbf{p}_2,\mathbf{p}_3\rangle_1)>0.$$
Let $\{\mathbb{X}_1,\ldots,\mathbb{X}_d\}$ be the set of cross-ratios associated with an ordered m-tuple $\mathfrak{p}=(p_1,\ldots,p_m)$ of distinct points in $\partial \mathbf{H}_{\mathbb{H}}^n$, where $d=m(m-3)/2$. We observe that $$\{\mathbb{X}_1,\ldots,\mathbb{X}_d,\mathbb{A_H}\}$$
is a parameter system which is invariant under the action of $PSp(n,1)$. This system with a unit pure quaternion $\mathfrak{u}$ and a positive number $r$ forms the minimal parameter system which uniquely determines the $PSp(n,1)$-congruence classes of m-tuple $\mathfrak{p}=(p_1,\ldots,p_m)$, showed in Theorem \ref{Theorem 5}.
Then, we have a characterization of the special-Gram quaternionic matrices associated with ordered m-tuple of distinct points in $\partial \mathbf{H}_{\mathbb{H}}^n$. This characterization gives all the relationships between our invariants, showed in Proposition \ref{Proposition4}, such that we can describe the Moduli space of the space $\mathcal{F}_1(n,m)$.

Our main results of this paper is concluded as following:
\begin{theorem1}\quad\label{theorem new1}
Let
$$\P_1(n,m)=\big\{w\in\mathbb{H}^d\times \mathfrak{sp}(1)\times\mathbb{R}^2:D^\star_{\mathrm{I}^{s}_{m-2}}(w)\geq0,\;s\leq n-1;\;
D^\star_{\mathrm{I}^{s}_{m-2}}(w)=0,\;s>n-1\big\},$$
where $d=m(m-3)/2$.
Then the space $\mathcal{F}_1(n,m)$ is homeomorphic to the set $\P_1(n,m)$.
\end{theorem1}

Therefore, the Moduli space for $\mathcal{F}_1(n,m)$ is described as $\P_1(n,m)$.

Then, we apply the above theorem to describe the Moduli space of the representation family of groups into $PSp(2,1)$. We have the following theorem.
\begin{theorem1}\label{Corollary new1}
Let $H$ be a finitely generated group with a fixed ordered set of generators, $\mathds{R}\mathrm{ep}^{o}_{\textit{lox}}(H,PSp(2,1))$ be the discrete, faithful, totally loxodromic representation family of $H$ into $PSp(2,1)$.
Then $\mathds{R}\mathrm{ep}^{o}_{\textit{lox}}(H,{\rm PSp}(2,1))$ can be identified with an open subset of $$\P_1(2,2k)\times\Upsilon^k\times \Theta^{2k},$$
where
$$\Upsilon=\{r\in\mathbb{R}\mid r>0,\;r\neq1\},\quad \Theta=[0,\pi].
$$
\end{theorem1}


Our paper is organised as follows.
In \S 2 we recall some basic facts in quaternion and quaternionic hyperbolic geometry.
In \S 3 we give a characterization of special-Gram quaternionic matrices of ordered m-tuple of distinct points in $\partial \mathbf{H}_{\mathbb{H}}^n$.
In \S 4 we construct and describe the moduli space and prove the main results. Then we apply our results to study of the Moduli space of discrete, faithful, totally loxodromic representation family of groups into $PSp(2,1)$.

When this paper was finished, soon we learnt that Cao \cite{Cao2} has also considered similar problem, but he use the different method and have different results from ours. We give more concrete discussion for the problem in terms of cross-ratios invariants.
\section{Preliminaries}
In this section, we briefly review some facts about quaternions and quaternionic hyperbolic space. We also introduce Moore 's determinant on quaternonic Hermitian matrices.
\subsection{Quaternions and quaternonic matrices}
Let $\mathbb{R}$, $\mathbb{C}$, $\mathbb{H}$ denote the sets of real, complex and quaternion numbers, respectively.
We write $$\mathbb{H}=\{q=t+ix+jy+kz~|~t,x,y,z\in \mathbb{R}\},$$ where $$i^2=j^2=k^2=-1,~ij=k=-ji,~jk=i=-kj,~ki=j=-ik.$$
Any $q\in \mathbb{H}$ can also be written in the form $q=z_1 + jz_2$, where $z_1,z_2\in \mathbb{C}$. Note that $zj=j\overline{z}$ for $z\in \mathbb{C}$, since the center of $\mathbb{H}$ is $\mathbb{R}$. Conjugation and modulus is given respectively by $$\overline{q}=\overline{t+ix+jy+kz}=t-ix-jy-kz,\quad|q|=\sqrt{q\overline{q}}=\sqrt{t^2+x^2+y^2+z^2}.$$ Then
$$\overline{q_1 q_2}=\overline{q_2}~\overline{q_1}\quad\text{for}\quad q_1,q_2\in\mathbb{H}.$$
$$\Re(q)=\frac{1}{2}(q+\overline{q})=t,$$
$$\Re(q_1q_2)=\Re(q_2q_1)=\Re(\overline{q}_1\overline{q}_2)=\Re(\overline{q}_1\overline{q}_2)
.$$
A pure quaternion is of the form $${\rm Pu}(q)=\frac{1}{2}(q-\overline{q})=ix+jy+kz$$ and  a unit quaternion is of the form $${\rm Un}(q)=\frac{q}{|q|}\in S,$$
where $S$ is the unit sphere in $\mathbb{H}$. Let $\mathfrak{sp}(1)$ is the Lie algebra of Lie group $Sp(1)$. Then $$\mathfrak{sp}(1)=\mathrm{Pu}(\mathbb{H})\cap S$$ will be the set of all unit pure quaternion.
Any quaternion is expressible as a power of a pure quaternion and it can be also  written in the form $$q=|q|e^{\mathfrak{u}\theta}:=|q|(\cos\theta + \mathfrak{u}\sin\theta),$$
where $\mathfrak{u}\in \mathfrak{sp}(1)$ and $\theta\in[0,\pi]$.

Next, we introduce the Moore 's determinant of a Hermitian quaternionic $n\times n$ matrix $M$. Such a matrix has quaternionic entries and  satisfies $M^*=M$, where $M^*$ is  the conjugate transpose of $M$. There are several kinds of determinants associated to a quaternionic matrix. Moore 's determinant is the one which is convenient for our purposes. For more information about Moore 's determinant, see \cite{E.H.Moore7}.

Moore 's determinant was first studied by E. H. Moore and we denote it by $\mathrm{det}_{\mathbb{M}}$. The following definition was given in \cite{E.H.Moore7}:
\begin{definition}\quad
For any Hermitian $n\times n$ quaternionic matrix $M$, the \textit{Moore's determinant} of $M$ is given by
$$\mathrm{det}_{\mathbb{M}}(M)=\sum_{\sigma\in s_n}\varepsilon(\sigma)m_{n_{11}n_{12}}\cdots m_{n_{1l_1}n_{11}}m_{n_{21}n_{22}}\cdots m_{n_{rl_r}n_{r1}},$$
where $S_n$ is group of permutations.
\end{definition}
Note that any $\sigma\in S_n$ can be written as a product of disjoint cycles. We permute each cycle until the leading index is the smallest and  then sort the cycles in increasing order according to the first number of each cycle. That is
 $$\sigma=(n_{11}\cdots n_{1l_1})(n_{21}\cdots n_{2l_2})\cdots (n_{r1}\cdots n_{rl_r}),$$
where for each $i$, we let $n_{i1}<n_{ij}$ for all $j>1$, $n_{11}<n_{21}\cdots <n_{r1}$ and
 $\varepsilon(\sigma)=(-1)^{n-r}$.
In fact, we can arrange the cycles in decreasing order according to initial index, used in \cite{Aslaksen15}.

We recall the following properties which was proved in \cite{Alesker29}.

\begin{lemma}\quad\label{Lemma1}
Let the quaterninic matrix $$M=\left(
        \begin{array}{cc}
          M_1 & \mathbf{0} \\
          \mathbf{0} & M_2 \\
        \end{array}
      \right)
,$$
where $M=M^*$, $M_1=M_{1}^{*}$, $M_2=M_{2}^{*}$. Then $$\mathrm{det}_{\mathbb{M}}(M) = \mathrm{det}_{\mathbb{M}}(M_1) \mathrm{det}_{\mathbb{M}}(M_2).$$
\end{lemma}

\begin{lemma}\quad\label{Lemma2}
(i) The Moore 's determinant of any  Hermitian complex matrix when this is considered as  Hermitian quaternionic matrix is equal to its usual determinant.
(ii) $\mathrm{det}_{\mathbb{M}}(U^*MU)=\det_{\mathbb{M}}(M)\det_{\mathbb{M}}(U^*U)$ for any Hermitian quaternionic matrix $M$ and any quaternionic matrix $U$.
\end{lemma}
Next is the quaternionic version of Sylvester's principle of inertia. It was studied carefully by Moore in \cite[P. 191-197]{Barn.Moore28}. It can also be derived by Lemma \ref{Lemma2} and the Claim 1.1.7 of \cite{Alesker29} and \cite{Zhang1997}. Note that the complex version was studied by
J.R. Parker in \cite{Parker2007}.

\begin{theorem}\quad[Sylvester's principle~of inertia] \label{theorem1}
The signature of a Hermitian quaternionic matrix is independent of the means of finding it. In particular, if $H_1$ and $H_2$ are two $k\times k$ matrices with the same signature then there exists a $k\times k$ matrix $U$ so that $H_2=U^*H_1U$.
\end{theorem}

The quaternionic generalization of the standard Sylvester criterion was given in Theorem 1.1.13 of \cite[P. 10]{Alesker29}. We restate it as follows:
\begin{theorem}\quad[Sylvester's criterion]\label{Theorem 3}
A Hermitian quaternionic matrix $M$ is positive definite if and only if its leading principal minors are all positive.

\end{theorem}

An analogous theorem holds for characterizing positive semi-definite Hermitian quaternionic matrices, which is showed in following theorem:

\begin{theorem}\quad\label{Thorem new2}
A Hermitian quaternionic matrix $M$ is positive semi-definite if and only if its leading principal minors are all nonnegative.
\end{theorem}
\begin{proof}
We first prove necessity. Suppose that $A$ is positive semi-definite and the matrix $A_k$ associated with the principal minors of order $k$ is given by
$$A_k=\left(
        \begin{array}{ccc}
          a_{i_1i_1} & \cdots & a_{i_1i_k} \\
          \vdots & \ddots & \vdots \\
          a_{i_ki_1} &\cdots & a_{i_ki_k} \\
        \end{array}
      \right).$$
Let $Y^* A Y$ and $X^* A_k X$ be the quadratic forms with matrices $A$ and $A_k$, respectively. For any $X_0=(b_{i_1},b_{i_2},\ldots,b_{i_k})^*\neq\mathbf{0}$, we have $Y_0=(c_1,\ldots,c_n)^*$, where $j=1,2,\ldots,n$ and if $j=i_1,i_2,\ldots,i_k$ then $c_j=b_j$, else $c_j=0$.

Since $A$ is positive semi-definite ${Y_0}^*AY_0\geq0$. Thus we obtain ${X_0}^*A_kX_0\geq0$ (i.e. $X^*A_kX$ is positive semi-definite). So there exists a non-degenerate matrix $T_k$, satisfying
$$
\mathrm{det}_{\mathbb M}({T_k}^*A_kT_k)
=\mathrm{det}_{\mathbb M}(A_k)\mathrm{det}_{\mathbb M}({T_k}^*T_k)\geq0.$$

Next, we prove sufficiency. Suppose that the principal minors of $A$ are all non-negative, choose the $k$-th leading principal minor and let $A_k$ be its corresponding matrix, given by
$$
\left(
  \begin{array}{cccc}
           a_{11} & a_{12} & \cdots & a_{1k} \\
           a_{21} & a_{22} & \cdots & a_{2k}\\
        \vdots & \vdots & \ddots & \vdots \\
        a_{k1} & a_{k2} & \cdots & a_{kk} \\
  \end{array}
\right),
$$
where $k=1,2,\cdots,n$. Then we have
$$
A_k+\lambda \mathrm{I}_k=
\left(
\begin{array}{cccc}
           a_{11}+\lambda & a_{12} & \cdots & a_{1k} \\
           a_{21} & a_{22}+\lambda & \cdots & a_{2k}\\
        \vdots & \vdots & \ddots & \vdots \\
        a_{k1} & a_{k2} & \cdots & a_{kk}+\lambda \\
  \end{array}
\right).
$$
According to Proposition 1.1.11 of \cite{Alesker29} we observe that
$$
\mathrm{det}_{\mathbb M}(A_k+\lambda \mathrm{I}_k)=\lambda^k+p_1\lambda^{k-1}+\cdots+p_{k-1}\lambda+p_k,
$$
where $p_i$ is the sum of all principal minors of $A_k$. By the assumption that all principal minors of $A$ are non-negative, we find that $p_i\geq0$.
Thus if $\lambda>0$ then $\mathrm{det}_{\mathbb M}(A_k+\lambda \mathrm{I}_k)>0$. That is to say, by Theorem \ref{Theorem 3}, if $\lambda>0$ then $A_k+\lambda \mathrm{I}_k$ is a positive-definite matrix.

Now, suppose that $A$ is not positive semi-definite matrix, then there exists a non-zero vector
$X_0=(b_1,b_2,\ldots,b_n)^*$ such that $$
{X_0}^*AX_0=-c,\quad c>0.
$$
Let $$\lambda=\frac{c}{{X_0}^*X_0}=\frac{c}{{b_1}^2+\cdots+{b_n}^2}>0.$$
Then
$$
{X_0}^*(A+\lambda \mathrm{I})X_0={X_0}^*AX_0 + {X_0}^*\lambda \mathrm{I} X_0=\lambda{X_0}^*X_0+{X_0}^*AX_0=c-c=0
,$$
 a contradiction.
\qed
\end{proof}

The rank of a quaternionic matrix was studied by Moore and Barnard in \cite[P. 64-70]{Barn.Moore28}, see  also \cite[P. 43]{Zhang1997}. We follow the former for our definition:
\begin{definition}\quad
Let $A$ be a finite but non-zero quaternionic matrix, then the \textit{rank} of $A$, denoted by $\mathrm{rank}(A)$, is defined to be the order of a maximal non-singular minor. If $A=\mathbf{0}$ then $\mathrm{rank}(A)=0$.
\end{definition}
\begin{remark}\quad
For each finite quaternionic matrix $A$ the rank is uniquely determined. It is equal to the number of columns (or rows) of $A$ in a maximal set of left linear independent columns (or rows) of $A$. Consequently, a finite matrix is non-singular if and only if the matrix is square and its rank is equal to its order.
\end{remark}
It could be checked directly that the quaternionic matrices $A$, $A^*$, $AA^*$ and $A^*A$ have the same rank.
\subsection{Quaternionic hyperbolic space}

 In this section, we shortly review  quaternionic hyperbolic space. For more information, see \cite{KGongopadhyay19,I.Kim17,KimI.18,Parker2007}.

Let the Hermitian quaternionic matrices
\begin{align*}
  H_1&=\left(
     \begin{array}{ccc}
       1 & \mathbf{0 }& 0 \\
       \mathbf{0} & \mathrm{I}_{n-1} & \mathbf{0} \\
       0 & \mathbf{0} & -1 \\
     \end{array}
   \right)
 ,~
 H_2=\left(
   \begin{array}{ccc}
     0 & \mathbf{0} & 1 \\
     \mathbf{0} & \mathrm{I}_{n-1} & \mathbf{0} \\
     1 & \mathbf{0} & 0 \\
   \end{array}
 \right)
 ,\\
   H_3&=\left(
     \begin{array}{ccc}
       -1 & \mathbf{0 }& 0 \\
       \mathbf{0} & \mathrm{I}_{n-1} & \mathbf{0} \\
       0 & \mathbf{0} & 1 \\
     \end{array}
   \right)
   ,~
   H_4=\left(
    \begin{array}{ccc}
      0 & -1 & 0 \\
      -1 & 0 & \mathbf{0} \\
      0 & \mathbf{0} & \mathrm{I}_{n-1} \\
    \end{array}
  \right).
\end{align*}
Denote by $\mathbb{H}^{n,1}$ the $(n+1)$-dimensional $\mathbb{H}$-vector space equipped with the Hermitian form of signature $(n,1)$ given by $\Phi_i(\mathbf{z},\mathbf{w})=\langle\mathbf{z},\mathbf{w}\rangle_i=\mathbf{w}^*H_i\mathbf{z}$. Here, $\mathbf{z,w}$ are column vectors in $\mathbb{H}^{n,1}$ and $i=1,2,3,4$. Let $\mathbb{P}$ denote a natural right projection from $ \mathbb{H}^{n,1}\setminus\{0\}$ onto the quaternionic projective space $\mathbb{HP}^n $.
Define subsets $V_-^i$, $V_0^i$, $V_+^i$ of $V=\mathbb{H}^{n,1}\setminus \{0\}$ as follows:
\begin{eqnarray*}
&&
V_-^i=\{V:\langle\mathbf{z,z}\rangle_i < 0\},\\
&&
V_{0}^i=\{V:\langle\mathbf{z,z}\rangle_i = 0\},\\
&&
V_+^i=\{V:\langle\mathbf{z,z}\rangle_i > 0\}.
\end{eqnarray*}
We say that $\mathbf{z}\in \mathbb{H}^{n,1}$ is negative, isotropic or positive if $\mathbf{z}$ is in $V_-$, $V_0$ or $V_+$, respectively. Motivated by relativity, these are sometimes called time-like, light-like and space-like. Their projections to $\mathbb{HP}^n $ are called negative, isotropic and positive points, respectively.

The quaternionic hyperbolic $n$-space $\mathbb{HP}^n$ is then defined by $\mathbf{H}_{\mathbb{H}}^n=\mathbb{P}(V_-^i)$ and its boundary is $\partial\mathbf{H}_{\mathbb{H}}^n=\mathbb{P}(V_0^i)$.
When $i=1~\mathrm{or}~3$, we have a ball model for quaternionic hyperbolic space and when $i=2$ or $4$, we have a Siegel domain model. We may pass between them using a corresponding Cayley transform.

The metric on $\mathbf{H}_{\mathbb{H}}^{n}$ is given by
$$\mathrm{ds}^2=\frac{-4}{\langle\mathbf{z},\mathbf{z}\rangle_i^2}
=\left(
   \begin{array}{cc}
     \langle\mathbf{z},\mathbf{z}\rangle_i & \langle\mathrm{d}\mathbf{z},\mathbf{z}\rangle_i \\
     \langle\mathbf{z},\mathrm{d}\mathbf{z}\rangle_i & \langle\mathrm{d}\mathbf{z},\mathrm{d}\mathbf{z}\rangle_i \\
   \end{array}
 \right)
.$$
Equivalently, it is given by the distance function $\rho$ where
$$\cosh^2\left(\frac{\rho(\mathbb{P}\mathbf{z},\mathbb{P}\mathbf{w})}{2}\right)=
\frac{\langle\mathbf{z},\mathbf{w}\rangle_i\langle\mathbf{w},\mathbf{z}\rangle_i}
{\langle\mathbf{z},\mathbf{z}\rangle_i\langle\mathbf{w},\mathbf{w}\rangle_i}.$$
We consider the Lie group ${\rm Sp}(n,1)=\{M\in GL_{n+1}(\mathbb{H}):M^*H_iM=H_i\}$. The isometry group of $\mathbf{H}_{\mathbb{H}}^{n}$ is the group ${\rm PSp}(n,1)={\rm Sp}(n,1)/\pm\mathrm{I}_{n+1}$. This group is a non-compact real semi-simple Lie group.

We recall the following Proposition which was showed in \cite[P. 52]{Chen1974}:
\begin{proposition}\quad\label{Proposition1}
${\rm Sp}(n,1)$ acts transitively on $V_-$ and doubly transitively on $V_0$.
\end{proposition}
According to the above Proposition, we conclude:
\begin{proposition}\quad\label{Proposition2}
If  $\mathbf{z}_1,\mathbf{z}_2$ are two distinct isotropic vectors in $\mathbb{H}^{n,1}$ then we have $\langle \mathbf{z}_1,\mathbf{z}_2\rangle_i\neq 0$.
\end{proposition}
\section{A characterization of special-Gram quaternionic matrices}
In this section, we consider the quaternionic hyperbolic space realized by the ball model related to the quadratic form $H_1$. An analogous method was used in the complex case in \cite{Cunha5}. Here, we also have to take into account the non-commutativity of quaternions. 
\subsection{Special-Gram quaternionic matrix}
Let $\mathfrak{p}=(p_1,\ldots,p_m)$ be an ordered $m$-tuple of pairwise distinct isotropic points in $\partial\mathbf{H}_\mathbb{H}^n$, $n\ge 2$ and let  $\tilde{\mathfrak{p}}=(\mathbf{p}_1,\ldots,\mathbf{p}_m)$ be its lift:  $\mathbf{p}_i\in \mathbb{H}^{n,1}$.

\begin{definition}\quad\label{Dfinition2}
The {\it special-Gram quaternionic matrix} associated to $\mathfrak{p}$ is $$
G=G(\tilde{\mathfrak{p}})=(g_{ij})=(\langle\mathbf{p}_j,\mathbf{p}_i\rangle_1),$$
where we assume that $\mathbf{p}_1$ is the standard lift (see \cite{Parker2007}) of $p_1$.
\end{definition}

Note that $g_{ij}$ is zero for $i=j$ and non-zero for $i\neq j$, by Proposition $\ref{Proposition2}$.
It is straightforward to check that special-Gram quaternionic matrices are invariant under the action of the isometry group, i.e. $$G(\tilde{\mathfrak{p}})=G(T(\tilde{\mathfrak{p}}))=(T(\mathbf{p}_1),\ldots ,T(\mathbf{p}_m)),~T\in {\rm PSp}(n,1).
$$

Still, $G$ depends on the choice of lifts $\mathbf{p}_i$ for $i=2,\ldots,m$.

If  $D= {\rm diag}(1,\lambda_2,\ldots,\lambda_m)$ is the diagonal quaternionic $m\times m$ matrix with quaternionic entries $\lambda_i\neq 0$ and $\tilde{\mathfrak{p}}^\prime=(\mathbf{p}_1,\mathbf{p}_2\lambda_2,\ldots,\mathbf{p}_m\lambda_m)$, then we have
$\widehat{G}=G(\tilde{\mathfrak{p}}^\prime)=(\langle p_j\lambda_j,p_i\lambda_i\rangle_1)=(\bar{\lambda}_i\langle p_j,p_i\rangle_1\lambda_j)$, in other words $\widehat{G}=D^*GD$.

We give the following definition of equivalence for special-Gram quaternionic matrices.
\begin{definition}\quad\label{definition 3}
Two special-Gram quaternionic matrices $M$ and $\widehat{M}$ are \textit{equivalent}, if there exists a non-singular diagonal matrix $D= {\rm diag}(1,\lambda_2,\ldots,\lambda_m)$ such that $\widehat{M}=D^*MD$.
\end{definition}

Let $\mathbf{p}_1$ be the standard lift of $p_1$, we will find that each ordered $m$-tuple $\mathfrak{p}$ of pairwise distinct points in $\partial\mathbf{H}_\mathbb{H}^n$ is assigned an equivalence class of special-Gram quaternionic matrices. Let $G$ and $\widehat{G}$ are two equivalence special-Gram quaternionic matrices associated with an $m$-tuple $\mathfrak{p}$. By Lemma \ref{Lemma2}, we have $$\mathrm{det}_{\mathbb M}(\widehat{G})=\lambda~\mathrm{det}_{\mathbb M}(G),$$
where $\lambda>0$. Note that the sign of $\mathrm{det}_{\mathbb M}(G)$ is independent of the chosen lifts $\mathbf{p}_i$.

We have the following proposition, compared with Proposition 2.1 of \cite{Cunha5}:
\begin{proposition}\quad\label{Prop}
Let $\mathfrak{p}=(p_1,\ldots,p_m)$ be an ordered $m$-tuple of pairwise distinct points in $\partial\mathbf{H}_\mathbb{H}^n$. Then the equivalence class of a special-Gram quaternionic matrix associated to $\mathfrak{p}$ contains a unique matrix $G(\mathfrak{p})=(g_{ij})=(\langle\mathbf{p}_j,\mathbf{p}_i\rangle_1)$ with $g_{ii}=0$, $g_{1j}=1$ for $j=2,\ldots,m$.
\end{proposition}
\begin{proof}
Let $\tilde{\mathfrak{p}}=(\mathbf{p}_1,\ldots,\mathbf{p}_m)$ be a lift of $\mathfrak{p}=(p_1,\ldots,p_m)$. We know that $g_{ij}\neq0$ whence $i\neq j$ due to the fact that $p_i$ are distinct and isotropic.
We re-scale $\mathbf{p}_j$ appropriately, replacing $\mathbf{p}_j$ by $\mathbf{p}_j\lambda_j$, to obtain that $g_{1j}=1$ where
$\lambda_1=1$ and $\lambda_j=\langle \mathbf{p}_j,\mathbf{p}_1\rangle_1^{-1}$ with $j=2,\ldots,m$.

Let $\tilde{\mathfrak{p}}^\prime=(\mathbf{p}_1,\mathbf{p}_2\mu_2,\ldots,\mathbf{p}_m\mu_m)$ is the another lift of $\mathfrak{p}=(p_1,\ldots,p_m)$. We start from $\tilde{\mathfrak{p}}^\prime$ to find the matrix $(g^\prime_{ij})$, where $(g^\prime_{ij})$ is in the equivalence class of a special-Gram quaternionic matrix associated to $\mathfrak{p}$ with $g^\prime_{ii}=0$ and $g^\prime_{1j}=1$ for $j=2,\ldots,m$. Then we will find  $(g_{ij})=(g^\prime_{ij})$. The uniqueness is clear.
\qed
\end{proof}

\begin{remark}\quad\quad
 If we choose arbitrary lift for $p_1$, then two different lifts of $\mathfrak{p}=(p_1,\ldots,p_m)$, $\tilde{\mathfrak{p}}=(\mathbf{p}_1,\ldots,\mathbf{p}_m)$ and $\tilde{\mathfrak{p}}^\prime=(\mathbf{p}_1\mu_1,\mathbf{p}_2\mu_2,\ldots,\mathbf{p}_m\mu_m)$, can't deduce to the same one matrix we need, i. e. $(g_{ij})\neq(g^\prime_{ij})$. In fact, we may deduce the matrix $(g_{ij})$ by different method from the one used in Proposition \ref{Prop} and observe that $g^\prime_{lj}= g_{lj}/|\mu_1|^2 $ for some fixed $l$. Hence $(g_{ij})=(g^\prime_{ij})$ only if $|\mu_1|=1$.
\end{remark}

The unique matrix defined in Proposition \ref{Prop} is called a \textit{normalized special-Gram quaternionic matrix} and is denoted by
$\underline{G}(\mathfrak{p})$ or simply  by $\underline{G}$ if there is no danger of confusion. It follows that there is a correspondence between the space of ordered $m$-tuples of pairwise distinct isotropic points and the space of normalized special-Gram quaternionic matrices.
\subsection{The Characterization of special-Gram quaternionic matrices}
In this section, we discuss some properties of special-Gram quaternionic matrices associated with m-tuple of isotropic pairwise distinct points, $m>1$.

The Hermitian form $\Phi_1(\mathbf{z},\mathbf{w})$ on a subspace $W\subset{\mathbb{H}^{n,1}}$ is \textit{degenerate}, if there exists a non-zero vector $w\in W$ such that $\langle w,v\rangle_1=0$ for all $v\in W$. We call the subspace $W$  \textit{degenerate}, otherwise \textit{non-degenerate}. Clearly, $w$ must lie in $V_0$.

\begin{lemma}\quad\label{Theorem2}
Let $W$ be a degenerate subspace of ${\mathbb{H}^{n,1}}$ and $$R(W)=\{\mathbf{w}\in W\mid\langle \mathbf{w},\mathbf{{u}}\rangle_1,\forall \mathbf{{u}}\in W\}.$$
Then $R(W)$ is a line in $V_0$.
\end{lemma}
\begin{proof}
Let $k={\rm dim}(W)$. If $k=1$ then the proof is completed. So we suppose $2\leq k\leq n+1$. We will prove that $R(W)$ can not have two linearly independent isotropic vectors. Supposing the contrary, i.e., there exist two such vectors $\mathbf{w}_1,\mathbf{w}_2 \in R(W)$. Then on one hand we have $\langle\mathbf{w}_1,\mathbf{w}_2\rangle_1=0$ by definition.  On the other hand, since $\mathbf{w}_1,\mathbf{w}_2\in V_0$, we have that  $\langle\mathbf{w}_1,\mathbf{w}_2\rangle_1\neq0$ by Proposition $\ref{Proposition2}$. This completes the proof.
\qed
\end{proof}

Lemma \ref{Theorem2} implies the following theorem.
\begin{theorem}\quad\label{Thorem new3}
Let $W\subset{\mathbb{H}^{n,1}}$.
For the Hermitian form $\Phi_1$ (but as well as for all other $\Phi_i$, $i=2,3,4$ as well) acting  on ${\mathbb{H}^{n,1}}$, the restriction $\Phi_1\mid W$ loses at most one dimension and the Witt index of a totally isotropic subspaces is 1.
\end{theorem}

We recall the following definition which was given in \cite[P. 52 ]{Chen1974}:
\begin{definition}\quad
A subspace $W\subset{\mathbb{H}^{n,1}}$ is called \textit{hyperbolic} if the restriction $\Phi_1\mid W$ is non-degenerate and indefinite; it is \textit{elliptic} if $\Phi_1\mid W$ is positive definite; and it is \textit{parabolic} if $\Phi_1\mid W$ is degenerate.
\end{definition}

Let $W$ be a $(k+1)$-dimensional subspace of $\mathbb{H}^{n,1}, 1\leq k\leq n$ and denote its signature by $(n_+,n_-,n_0)$, where $n_+$ (resp. $n_-$, $n_0$) is the number of positive (reps. negative, zero) eigenvalues of the Hermitian matrix of $W$. Then Theorem \ref{Thorem new3} implies the following corollary:
\begin{corollary}
The restriction of the Hermitian product $\Phi_1$ on $\mathbb{H}^{n,1}$ to $W$ only has signature $(k,1,0)$, $(k+1,0,0)$, $(k-1,1,0)$ or $(k,0,0)$.
\end{corollary}

Let $\mathfrak{p}=(p_1,\ldots,p_m)$ be an ordered $m$-tuple of pairwise distinct points in $\mathbb{HP}^n$ and let $L(\mathbf{p}_i,\ldots,\mathbf{p}_m)\subset{\mathbb{H}^{n,1}}$ be the subspace spanned by the distinct vectors $\mathbf{p}_i,\ldots,\mathbf{p}_m$, where as usual $\mathbf{p}_i\in\mathbb{H}^{n,1}$ is the lift of the isotropic point $p_i$, $i=1,\dots,m$. Moreover assume that ${\rm dim}(L(\mathbf{p}_i,\ldots,\mathbf{p}_m))=k+1$, then the following three cases exhaust all possibilities, compared with Cunha \cite{Cunha5}:
\begin{enumerate}
  \item $L(\mathbf{p}_i,\ldots,\mathbf{p}_m)$ is hyperbolic if it has signature $(k,1,0)$, where $1\leq k\leq n$;
  \item $L(\mathbf{p}_i,\ldots,\mathbf{p}_m)$ is elliptic if it has signature $(k+1,0,0)$, where $1\leq k\leq n$;
  \item $L(\mathbf{p}_i,\ldots,\mathbf{p}_m)$ is parabolic if it has signature $(k,0,1)$ or $(k-1,1,1)$, where $1\leq k\leq n-1$.
\end{enumerate}

Let $G(\tilde{\mathfrak{p}})=(g_{ij})=(\langle \mathbf{p}_j,\mathbf{p}_i\rangle_1)$ be the special-Gram matrix associated to an $m$-tuple $\mathfrak{p}=(p_1,\ldots,p_m)$ of pairwise distinct points in the quaternionic projective space $\mathbb{HP}^n$, where $\mathbf{p}_i$ is the lift of $p_i\in \mathbb{HP}^n$, $i=1,\dots.m$. By Corollary 6.2 of \cite[P. 41]{Zhang1997} and Theorem $\ref{theorem1}$, the following holds:
$$G(\tilde{\mathfrak{p}})=\left(
                            \begin{array}{cccc}
                              \langle \mathbf{p}_1,\mathbf{p}_1\rangle_1 & \langle \mathbf{p}_2,\mathbf{p}_1\rangle_1 & \cdots& \langle \mathbf{p}_m,\mathbf{p}_1\rangle_1 \\
                              \langle \mathbf{p}_1,\mathbf{p}_2\rangle_1 & \langle \mathbf{p}_2,\mathbf{p}_2\rangle_1 & \cdots & \langle \mathbf{p}_m,\mathbf{p}_2\rangle_1 \\
                              \vdots & \vdots & \ddots & \vdots\\
                              \langle \mathbf{p}_1,\mathbf{p}_m\rangle_1 &\langle \mathbf{p}_2,\mathbf{p}_m\rangle_1 & \cdots  & \langle \mathbf{p}_m,\mathbf{p}_m\rangle_1 \\
                            \end{array}
                          \right)
=C^*JC,
$$
where $C\in {\rm GL}_m(\mathbb{H})$ and the Hermitian matrix $J={\rm diag}(\alpha_1,\ldots,\alpha_m)$, with $\alpha_i=\pm 1$ or $0$ and $G(\tilde{\mathfrak{p}})$ has the same signature as $J$, which is obtained by restricting $\Phi_1$ on $\mathbb{H}^{n,1}$ to the subspace $L(\mathbf{p}_i,\ldots,\mathbf{p}_m)$.

Therefore the signature of the special-Gram qutaternionic matrix $G(\tilde{\mathfrak{p}})$ may be given according to the type of $L(\mathbf{p}_i,\ldots,\mathbf{p}_m)$:
\begin{enumerate}
  \item If $L(\mathbf{p}_i,\ldots,\mathbf{p}_m)$ is hyperbolic, then $G(\tilde{\mathfrak{p}})$ has signature $(n_+,n_-,n_0)$ with $1\leq n_+\leq n$, $n_-=1$, and $n_++1+n_0=m$.
  \item If $L(\mathbf{p}_i,\ldots,\mathbf{p}_m)$ is elliptic, then $G(\tilde{\mathfrak{p}})$ has signature $(n_+,n_-,n_0)$ with $1\leq n_+\leq n$, $n_-=0$, and $n_++n_0=m$.
  \item If $L(\mathbf{p}_i,\ldots,\mathbf{p}_m)$ is parabolic, then $G(\tilde{\mathfrak{p}})$ has  signature $(n_+,n_-,n_0)$ with $1\leq n_+\leq n$, $1\leq n_0\leq n$, $n_-=0$ or $1$, and $n_++n_-+n_0=m$.
\end{enumerate}
The above conditions shall be called the \textit{signature conditions}.

We now focus on the case that $\mathfrak{p}=(p_1,\ldots,p_m)$ is an ordered $m$-tuple of pairwise distinct points in $\partial\mathbf{H}_{\mathbb{H}}^{n}$, with the lift $\tilde{\mathfrak{p}}=(\mathbf{p}_i,\ldots,\mathbf{p}_m)$. Since the Moore's determinant of the upper left hand $2\times2$-block of $G(\tilde{\mathfrak{p}})$ is $-1$, we see that $L(\mathbf{p}_i,\ldots,\mathbf{p}_m)$ always has signature with $n_-=1$. Thus any special-Gram quaternionic matrix associated to an $m$-tuple of pairwise distinct isotropic points has exactly signature $(n_+,n_-,n_0)$ with $n_-=1$, $1\leq n_+\leq n$, and $1+n_++n_0 =m$.
Then the following holds:
\begin{proposition}\quad\label{pro4}
Let $G=(g_{ij})$ be a Hermitian quaternionic matrix with $g_{ii}=0$, $g_{ij}\neq0$ for $i\neq j,~m>1$. Then $G$ is a special-Gram quaternionic matrix associated with some ordered m-tuple $\mathfrak{p}=(p_1,\ldots,p_m)$ of distinct isotropic points in $\partial\mathbf{H}_{\mathbb{H}}^{n}$ if and only if $\mathrm{rank}(G)\leq n+1$ and $G$ has the signature $(n_+,n_-,n_0)$ with $n_-=1$, $1\leq n_+\leq n$, and $1+n_++n_0 =m$.
\end{proposition}

The proof is similar to the proof of the Proposition 2.2 in \cite{Cunha5}. For completeness, we give it as follows:
\begin{proof}
Let $G=(g_{ij})$ be a Hermitian quaternionic matrix with $g_{ii}=0$, $g_{ij}\neq0$ for $i\neq j,~m>1$. According to Theorem \ref{theorem1}, there exists a matrix $U\in GL_m(\mathbb{H})$ such that $U^*GU=B$, where $B=(b_{ij})$ is the diagonal $m\times m$-matrix such that $b_{ii}=1$ for $1\leq i\leq n_+$, $b_{ii}=-1$ for $i=n_++1$ and $b_{ij}=0$ for all others. Now let $A=(a_{ij})$ be the $(n+1)\times m$-matrix such that $a_{ii}=1$ for $1\leq i\leq n_+,~a_{ii}=-1$ for $i=n_++1$ and $a_{ij}=0$ for all others. Then we easily obtain that $B=A^*H_1A$. Thus the $i$-th column vector of the matrix $AU^{-1}$ can be defined to be $\mathbf{p}_i$. We see that $\langle\mathbf{p}_j,\mathbf{p}_i\rangle_1=g_{ij}$. Finally, we find the isotropic points $p_i=\pi(\mathbf{p}_i)$, this is our desired.
\qed
\end{proof}

The following proposition is in \cite[P. 52]{Chen1974}:
\begin{proposition}\quad
Let $W$ be a subspace of $\mathbb{H}^{n,1}$. Each linear isometry of $W$ into $\mathbb{H}^{n,1}$ can be extended to an element of ${\rm Sp}(n,1)$.
\end{proposition}
We thus have:
\begin{proposition}\quad\label{pro6}
Let $\mathfrak{p}=(p_1,\ldots,p_m)$ and $\mathfrak{p^\prime}=(p^{\prime}_1,\ldots,p^{\prime}_m)$ be two $m$-tuples of pairwise distinct points in $\partial\mathbf{H}_{\mathbb{H}}^{n}$. Then $\mathfrak{p}$ and $\mathfrak{p}^\prime$ are congruent in ${\rm PSp}(n,1)$ if and only if their associated special-Gram quaternionic matrices are equivalent.
\end{proposition}

\begin{corollary}\label{Corollary1}
Let $\mathfrak{p}=(p_1,\ldots,p_m)$ and $\mathfrak{p^\prime}=(p^{\prime}_1,\ldots,p^{\prime}_m)$ be two $m$-tuples of pairwise distinct points in $\partial\mathbf{H}_{\mathbb{H}}^{n}$, and let $\underline{G}(\mathfrak{p})$ and $\underline{G}(\mathfrak{p}^\prime)$ be their normalized special-Gram quaternionic matrices. Then $\mathfrak{p}$ and $\mathfrak{p}^\prime$ are congruent in ${\rm PSp}(n,1)$ if and only if $\underline{G}(\mathfrak{p})=\underline{G}(\mathfrak{p}^\prime)$.
\end{corollary}

Let $\underline{G}(\mathfrak{p})=(g_{ij})$ be the matrix
$$\underline{G}(\mathfrak{p})=\left(
      \begin{array}{cccccc}
        0 & 1 & 1 & 1 & \cdots & 1 \\
        1 & 0 & g_{23} & g_{24} & \cdots & g_{2m}\\
        1 & \overline{g}_{23} & 0 & g_{34} & \cdots & g_{3m} \\
        1 & \overline{g}_{24} &\overline{g}_{34} & 0 & \cdots & g_{4m} \\
        \vdots & \vdots & \vdots & \vdots & \ddots & \vdots \\
        1 & \overline{g}_{2m} & \overline{g}_{3m} & \overline{g}_{4m} & \cdots & 0 \\
      \end{array}
    \right),
$$
where $g_{ij}\neq0$ for $i\neq j$.
Then the normalized special-Gram quaternionic matrix associated with an ordered $m$-tuple of pairwise distinct isotropic points is necessarily of the above form.
\begin{proposition}\quad\label{Proposition3}
Let $\underline{G}=(g_{ij})$ be a Hermitian quaternionic $m\times m$-matrix, $m>2$, with that $g_{ii}=0, g_{1j}=1$ for $j=2, \ldots, m$ and $g_{ij}\neq0$ for $i\neq j$. Then there exists a matrix in $GL_m(\mathbb{H})$ which transforms $\underline{G}$ into the matrix
\begin{equation*}
\left(
  \begin{array}{cc|c}
    0 & 1 & \mathbf{0} \\
    1 & 0 & \mathbf{0} \\ \hline
    \mathbf{0} & \mathbf{0} & G^{\star}\\
  \end{array}
\right),
\end{equation*}
where $G^{\star}$ is Hermitian quaternionic $(m-2)\times (m-2)$-matrix given by
$$G^{\star}=\left(
       \begin{array}{cccc}
         -(g_{23}+\overline{g}_{23}) & -\overline{g}_{23}-g_{24}+g_{34} & \cdots & -\overline{g}_{23}-g_{2m}+g_{3m}  \\
         -g_{23}-\overline{g}_{24}+\overline{g}_{34} & -(g_{24}+\overline{g}_{24}) & \cdots & -\overline{g}_{24}-g_{2m}+g_{4m} \\
         \vdots & \vdots & \ddots & \vdots \\
         -g_{23}-\overline{g}_{2m}+\overline{g}_{3m} & -g_{24}-\overline{g}_{2m}+\overline{g}_{4m} & \cdots & -(g_{2m}+\overline{g}_{2m}) \\
       \end{array}
     \right).
$$
\end{proposition}
\begin{proof}
Let $$T=\left(
         \begin{array}{cccccc}
           1 & 0 & -g_{23} & -g_{24} & \cdots & -g_{2m} \\
           0 & 1 & -1 & -1 & \cdots & -1 \\
           0 & 0 & 1 & 0 & \cdots & 0 \\
           0 & 0 & 0 & 1 & \cdots & 0 \\
           \vdots & \vdots & \vdots & \vdots & \ddots & \vdots \\
           0 & 0 & 0 & 0 & \cdots & 1 \\
         \end{array}
       \right)
$$
then
\begin{equation*}
T^*\underline{G}T=\left(
  \begin{array}{cc|c}
    0 & 1 & \mathbf{0} \\
    1 & 0 & \mathbf{0} \\ \hline
   \mathbf{ 0} & \mathbf{0} & G^{\star}\\
  \end{array}
\right).
\end{equation*}
\qed
\end{proof}

The matrix $G^{\star}$ is called the {\it associated} matrix to $\underline{G}$.
\begin{corollary}
Let $\underline{G}=(g_{ij})$ be a Hermitian quaternionic $m\times m$-matrix satisfying the conditions of Proposition $\ref{Proposition3}$ and let $G^{\star}$ be the associated matrix to $\underline{G}$. Then
$$
\det\nolimits_{\mathbb M}({\underline G})=\frac{-\det_{\mathbb M}(G^{\star})}{\det_{\mathbb M}(T^*T)},~\mathrm{and}~\mathrm{rank}(\underline{G})=\mathrm{rank}(G^{\star})+2.$$
\end{corollary}

The following Theorem follows by Proposition $\ref{pro4}$.
\begin{theorem}\quad\label{thm:G}
Let $\underline{G}=(g_{ij})$ be a Hermitian quaternionic $m\times m$-matrix, $m>2$, such that $g_{ii}=0$, $g_{1j}=1$ for $j=2,\ldots,m$, and $g_{ij}\neq0$ for $i\neq j$. Let $G^{\star}$ be the associated matrix to $\underline{G}$. Then $\underline{G}$ is a special-Gram quaternionic matrix associated to some ordered $m$-tuple $\mathfrak{p}=(p_1,\ldots,p_m)$ of pairwise distinct points in $\partial\mathbf{H}_{\mathbb{H}}^{n}$ if and only if $\mathrm{rank}(G^{\star})\leq n-1$ and $G^{\star}$ is positive semi-definite.
\end{theorem}

Using Theorem $\ref{Thorem new2}$, we may express Theorem \ref{thm:G} equivalently as:
\begin{theorem}\quad\label{Theorem 6}
Let $\underline{G}=(g_{ij})$ be a Hermitian quaternionic $m\times m$-matrix, $m>2$, such that $g_{ii}=0$, $g_{1j}=1$ for $j=2,\ldots,m$, and $g_{ij}\neq0$ for $i\neq j$. Let $G^{\star}$ be the associated matrix to $G$. Then $\underline{G}$ is a special-Gram quaternionic matrix associated to some ordered m-tuple $\mathfrak{p}=(p_1,\ldots,p_m)$ of pairwise distinct  points in $\partial\mathbf{H}_{\mathbb{H}}^{n}$ if and only if $\mathrm{rank}(G^{\star})\leq n-1$ and all principal minors of $G^{\star}$ are non-negative.
\end{theorem}

We call the conditions in Theorem \ref{Theorem 6} the \textit{Moore's determinant conditions}.
\section{The Moduli space}
\subsection{Invariants}
We first  recall Cartan's angular invariant and quaternionic cross-ratio. One may find more information in \cite{Apanasov2007,Cao,Platis2014}.

Let $\mathfrak{p}=(p_1,p_2,p_3)$ be an ordered triple of pairwise distinct  points in $\partial\mathbf{H}_{\mathbb{H}}^n$. The quaternionic Cartan's angular $\mathbb{A}_{\mathbb{H}}(\mathfrak{p})$ of $\mathfrak{p}$ is defined by
$$
\mathbb{A}_{\mathbb{H}}(\mathfrak{p})=
\arccos\frac{\Re(-\langle\mathbf{p}_1,\mathbf{p}_2,\mathbf{p}_3\rangle_1)}
{|\langle\mathbf{p}_1,\mathbf{p}_2,\mathbf{p}_3\rangle_1|},
$$
where $\mathbf{p}_i\in \mathbb{H}^{n,1}$ are the lifts of $p_i$, $\langle\mathbf{p}_1,\mathbf{p}_2,\mathbf{p}_3\rangle_1=\langle\mathbf{p}_1,\mathbf{p}_2\rangle_1
\langle\mathbf{p}_2,\mathbf{p}_3\rangle_1\langle\mathbf{p}_3,\mathbf{p}_1\rangle_1$ and $\Re(-\langle\mathbf{p}_1,\mathbf{p}_2,\mathbf{p}_3\rangle_1)>0$.

Rewrite $-\langle\mathbf{p}_1,\mathbf{p}_2,\mathbf{p}_3\rangle_1$ as $|\langle\mathbf{p}_1,\mathbf{p}_2,\mathbf{p}_3\rangle_1|e^{\mathfrak{u}\theta}$, where $\mathfrak{u}\in\mathfrak{sp}(1)$. Then $\theta=\mathbb{A}_{\mathbb{H}}(\mathfrak{p})$. Therefore, the quaternionic Cartan's angular  can also be given by:
$$\mathbb{A}_{\mathbb{H}}(\mathfrak{p})=\arg(-\langle\mathbf{p}_1,\mathbf{p}_2,\mathbf{p}_3\rangle_1).$$

One can prove that $0\leq\mathbb{A}_{\mathbb{H}}(\mathfrak{p})\leq\pi/2$
and $\mathbb{A}_{\mathbb{H}}(\mathfrak{p})$ independent of the chosen lifts and the order of three points.

We give the definition of quaternionic cross-ratio, following \cite{Cao}. See \cite{Platis2014} for more cross-ratio.

Let $\mathfrak{p}=(p_1,p_2,p_3,p_4)$ be a quadruple of pairwise distinct  points in $\partial\mathbf{H}_{\mathbb{H}}^n$. Their cross-ratio is defined by $$\mathbb{X}(p_1,p_2,p_3,p_4)=\langle \mathbf{p}_3,\mathbf{p}_1\rangle_1\langle\mathbf{p}_3,\mathbf{p}_2\rangle_1^{-1}\langle\mathbf{p}_4,\mathbf{p}_2
\rangle_1\langle\mathbf{p}_4,\mathbf{p}_1\rangle_1^{-1},$$
where $\mathbf{p}_i$ are the lifts of $p_i$.
Observe that $$\mathbb{X}(p_1\lambda_1,p_2\lambda_2,p_3\lambda_3,p_4\lambda_4)=\overline{\lambda}_1\langle \mathbf{p}_3,\mathbf{p}_1\rangle_1\langle\mathbf{p}_3,\mathbf{p}_2\rangle_1^{-1}\langle\mathbf{p}_4,\mathbf{p}_2
\rangle_1\langle\mathbf{p}_4,\mathbf{p}_1\rangle_1^{-1}{\overline{\lambda}_1}^{-1}.$$

Let now $\mathfrak{p}=(p_1,\ldots,p_m)$ be an $m$-tuple of pairwise distinct points in $\partial\mathbf{H}_{\mathbb{H}}^n$.
For short, let $$\mathbb{X}_{2j}=\mathbb{X}(p_1,p_2,p_3,p_j),~\mathbb{X}_{3j}=\mathbb{X}(p_1,p_3,p_2,p_j),~
\mathbb{X}_{kj}=\mathbb{X}(p_1,p_k,p_3,p_j),$$
where $m\geq4$, $4\leq k \leq m-1$, $k<j$. It is clear that the number of the above cross-ratios is equal to $m(m-3)/2$.

By direct computations, we have the following proposition:
\begin{proposition}\quad\label{Proposition4}
Let $\mathfrak{p}=(p_1,\ldots,p_m)$ be an $m$-tuple of pairwise distinct points in $\partial\mathbf{H}_{\mathbb{H}}^n$ and $\underline{G}(\mathfrak{p})=(g_{ij})$ be the normalized special-Gram quaternionic matrix of $\mathfrak{p}$. Then the following relations hold:
\begin{equation*}
\begin{aligned}
\mathbb{A}_{\mathbb{H}}&=\mathbb{A}_{\mathbb{H}}(p_1,p_2,p_3)=\mathrm{arg}(-\overline{g}_{23}),\\
\mathbb{X}_{2j}&=\mathbb{X}(p_1,p_2,p_3,p_j)=g_{23}^{~-1}g_{2j},\\
\mathbb{X}_{3j}&=\mathbb{X}(p_1,p_3,p_2,p_j)={\overline{g}_{23}^{~-1}}g_{3j},\\
\mathbb{X}_{kj}&=\mathbb{X}(p_1,p_k,p_2,p_j)=\overline{g}_{2k}^{~-1}g_{kj},\\
\end{aligned}
\quad\text{and}\quad
\begin{aligned}
g_{23}&=-re^{-\mathfrak{u}\mathbb{A}_{\mathbb{H}}},\\
g_{2j}&=
-re^{-\mathfrak{u}\mathbb{A}_{\mathbb{H}}}\mathbb{X}_{2j},\\
g_{3j}&=
-re^{\mathfrak{u}\mathbb{A}_{\mathbb{H}}}\mathbb{X}_{3j},\\
g_{kj}&=-r\overline{\mathbb{X}}_{2k}e^{\mathfrak{u}\mathbb{A}_{\mathbb{H}}}\mathbb{X}_{kj},
\end{aligned}
\end{equation*}
where $r=|g_{23}|$  and $\mathfrak{u}\in\mathfrak{sp}(1) $ and all indices are in accordance with the indices of the above-defined  cross-ratios.
\end{proposition}

According to Corollary \ref{Corollary1}, we have the following result which is useful for us to study the Moduli space:
\begin{theorem}\quad\label{Theorem 5}
The congruence class of $\mathfrak{p}$ in ${\rm PSp}(n,1)$ is uniquely determined by a positive number $r$ and a unit pure quaternion $\mathfrak{u}\in \mathfrak{sp}(1)$ and the invariants given the above $\mathbb{X}_{2j},\mathbb{X}_{3j},\mathbb{X}_{kj},\mathbb{A_H}$.
\end{theorem}
\subsection{Moduli space and proof of the theorem \ref{theorem new1}}
Let $\mathrm{I}^{s}_{n}=\{(i_1,\ldots,i_s)\;|\;1\leq i_1< i_2\ldots<i_s\leq n\}$ and let $A_{\mathrm{I^{s}_{n}}}$ denote the sub-matrix of an $n\times n$-matrix $A$, formed by choosing the elements of the original matrix from the rows whose indices are in $i_1,\ldots,i_s$ and columns whose indexes are in $i_1,\ldots,i_s$.

Let $\mathfrak{p}=(p_1,\ldots,p_m)$ be an $m$-tuple of pairwise distinct points in $\partial\mathbf{H}_{\mathbb{H}}^n$ and $\underline{G}(\mathfrak{p})=(g_{ij})$ be the normalized special-Gram quaternionic matrix of $\mathfrak{p}$. Let $G^\star$ be the associated $(m-2)\times (m-2)$-matrix to $\underline{G}(\mathfrak{p})$. The principal minors of $G^\star$ are ${\det_{\mathbb M}}(G^\star_{\mathrm{I}^{s}_{m-2}}).$

By Proposition $\ref{Proposition4}$, we treat ${\det_{\mathbb M}}(G^\star_{\mathrm{I}^{s}_{m-2}})$ as functions of $$(\mathbb{X}_{2j},\mathbb{X}_{3j},\ldots,\mathbb{X}_{kj},\mathfrak{u},\mathbb{A_H},r).$$
Identify $w=(q_1,\ldots,q_d,\mathfrak{u},t_1,t_2)$ with $(\mathbb{X}_{2j},\mathbb{X}_{3j},\ldots,\mathbb{X}_{kj},\mathfrak{u},\mathbb{A_H},r)$, where $ q_i\in\mathbb{H}$ is nonzero, $i=1,\ldots,d$ and $d=m(m-3)/2$.

We define the map $$D^\star_{\mathrm{I}^{s}_{m-2}}:\mathbb{H}^{m(m-3)/2}\times \mathfrak{sp}(1)\times\mathbb{R}^2\rightarrow \mathbb{R},$$
given by
$$w\mapsto D^\star_{\mathrm{I}^{s}_{m-2}}(w)={\det\nolimits_{\mathbb M}}(G^\star_{\mathrm{I}^{s}_{m-2}}).$$

Let $\mathcal{F}_1(n,m)$ be the quotient configuration space of ordered $m$-tuples of pairwise distinct points in $\partial\mathbf{H}_{\mathbb{H}}^n$, that is, the space of $m$-tuples cut by the action of $PSp(n,1)$. Let $[\mathfrak{p}]\in\mathcal{F}_1(n,m)$. Then, according to Theorem $\ref{Theorem 6}$ and Proposition $\ref{Proposition4}$, we define the map
$$\tau_1:\mathcal{F}_1(n,m)\to\mathbb{H}^{m(m-3)/2}\times \mathfrak{sp}(1)\times\mathbb{R}^2,$$
given by
$$[\mathfrak{p}]\mapsto w.$$
\begin{remark}\quad
The lower index 1 means that we use the Hermitian form $H_1$ for the entries of the special-Gram matrix. Likewise, the index 2, 3, 4.
\end{remark}
Let $$
\P_1(n,m)=\big\{w\in\mathbb{H}^d\times \mathfrak{sp}(1)\times\mathbb{R}^2:D^\star_{\mathrm{I}^{s}_{m-2}}(w)\geq0,\;s\leq n-1;\;
D^\star_{\mathrm{I}^{s}_{m-2}}(w)=0,\;s>n-1\big\},$$
where $w=(q_1,\ldots,q_d,\mathfrak{u},t_1,t_2)$, $0\neq q_i\in\mathbb{H}$, $\mathfrak{u}\in \mathfrak{sp}(1)$, $t_1\in[0,\pi/2]$, $t_2>0$ for all $i=1,\ldots,d=m(m-3)/2$.

Now let's prove our main result:

\begin{proof1}~

Our purpose is to find a homeomorphic map $\mathcal{F}_1(n,m)\to \P_1(n,m)$. It will be the map $\tau_1:\mathcal{F}_1(n,m)\to \P_1(n,m)$ as above.
Let $\P_1(n,m)$ be equipped with the topology inherited from $\mathbb{H}^{m(m-3)/2}\times \mathfrak{sp}(1)\times\mathbb{R}^2$. Hence we only need to prove that  $\tau_1$ is bijective.

Injectivity follows straightforwardly by Theorem $\ref{Theorem 5}$.
It is only necessary to show that the map $\tau_1$ is surjective.

If $w\in \P_1(n,m)$, we can construct a Hermitian quaternionic $m\times m$-matrix $\underline{G}=(g_{ij})$ with $g_{ii}=0$ and $g_{1j}=1$ for $j=2,\ldots,m$.
We see that  $w=(q_1,\ldots,q_d,\mathfrak{u},t_1,t_2)$ identify with $(\mathbb{X}_{2j},\mathbb{X}_{3j},\ldots,\mathbb{X}_{kj},\mathfrak{u},\mathbb{A_H},r)$. Then we can fix the other entries of $\underline{G}$ using Proposition $\ref{Proposition4}$.
If this $\underline{G}$ satisfies Moore's conditions, see Theorem $\ref{Theorem 6}$, then $\underline{G}$ is the normalized special-Gram quaternionic matrix for some ordered $m$-tuple of pairwise distinct  points $\mathfrak{p}$ in $\partial\mathbf{H}_{\mathbb{H}}^n$. In other words, $w$ uniquely corresponds to a point $[\mathfrak{p}]\in \mathcal{F}_1(n,m)$. This proves that $\tau_1$ is surjective.
\qed
\end{proof1}

We call $\P_1(n,m)$ the $moduli$ $space$ for $\mathcal{F}_1(n,m)$.  The following corollary is clear.
\begin{corollary}
$\mathcal{F}_1(n,m)$ is a real $2m^2-6m+5-\sum^{m-n-1}_{i=1}{m-2 \choose n-1+i}$ dimensional algebraic variety when $m>n+1$.
\end{corollary}

\begin{remark}\quad
$\mathcal{F}_1(n,m)$ is homeomorphic to $\mathcal{F}_i(n,m)$ since we may pass between two different Hermitian forms of the same signature by using Cayley transformations, where $i=2,3,4$. Therefore $\P_i(n,m)$ shares the same topology with $\P_1(n,m)$.
\end{remark}
\subsection{Moduli space for representation family of totally loxodromic groups and proof of the theorem \ref{Corollary new1}}
This section is motivated by the last section of \cite{Cunha5}. However, our results are different, even when we restrict ourselves to the complex hyperbolic space. Our target space is the deformation space of discrete, faithful, totally loxodromic and finitely generated groups in ${\rm PSp}(2,1)$.
\begin{definition}\quad
A subgroup of ${\rm PSp}(n,1)$ is called {\it totally loxodromic} if it comprises only loxodromic elements and the identity.
\end{definition}

The following important proposition comes from \cite[Corollary 4.5.4.]{Chen1974}.
\begin{proposition}\quad\label{Proposition6}
Let $G$ be a totally loxodromic subgroup of ${\rm Sp}(n,1)$ such that the quaternionic dimension $\mathrm{dim_{\mathbb{H}}}(M(G))$ of the smallest $G$-invariant totally geodesic submanifold $M(G)$, is even. Then $G$ is discrete.
\end{proposition}

Let $G_0$ be a totally loxodromic subgroup of ${\rm PSp}(n,1)$ and $\mathrm{dim_{\mathbb{H}}}~M(G_0)$ be even. By the preceding proposition, $G_0$ is discrete. For example, the totally loxodromic subgroups of ${\rm PSp}(2,1)$ that are generated by finite distinct generators without common fixed points. These subgroup are discrete subgroups of ${\rm PSp}(2,1)$.

Let $H=\langle h_1,h_2,\ldots,h_k\rangle$ be a finitely generated group with a fixed ordered set of generators $\langle h_1,h_2,\ldots,h_k\rangle$ and $G$ a topological group. The set of homomorphisms $\mathrm{Hom}(H,G)$ naturally sits inside $G^k$ via the evaluation map $f:\mathrm{Hom}(H,G)\to G^k$ given by $\rho\mapsto (\rho(h_1),\ldots,\rho(h_k))$. Hence, $\mathrm{Hom}(H,G)$ has an induced topology by $f$.

\begin{definition}\quad
The \textit{representation family} of discrete, faithful, totally loxodromic representations of $H$ into $G$ is
$$\mathds{R}\mathrm{ep}^{o}_{\textit{lox}}(H,G)=\left\{\rho\in \mathrm{Hom}(H,G):\rho~\mathrm{injective};~\rho(H)~\mathrm{discrete,~loxodromic}
\right\}.$$
\end{definition}

We endow $\mathds{R}\mathrm{ep}^{o}_{\textit{lox}}(H,G)$ with the topology of point-wise convergence. In this topology a sequence formed by the homomorphisms $\rho_j:H\to G,j=1,2,\ldots,$ converges to a homomorphism $\rho:H\to G$ if and only if for each $h\in H$ the sequence $\rho_1(h),\rho_2(h),\ldots$ converges to $\rho(h)$ in the topological group $G$.

In remainder of this section we assume $G=PSp(2,1)$.

The points of $\mathds{R}\mathrm{ep}^{o}_{\textit{lox}}(H,G)$ are identified with the $G$-conjugation equivalence classes of $G_0$, where $G_0=\rho(H)$. Let $G_0=\langle g_1,\ldots,g_k\rangle$ with $g_i=\rho(h_i)$, then  $g_i\neq g_j$ for all $i\neq j$.

We now describe the moduli space of $\mathds{R}\mathrm{ep}^{o}_{\textit{lox}}(H,G)$. First, we need to verify that the $G$-conjugation equivalence class of each loxodromic element $g_i$ can be uniquely determined by some parameters.

Cao and Gongopadhyay proved the following theorem on the conjugation classification for the complex and the quaternionic hyperbolic plane in \cite[Theorem 3.1 (i)]{Cao2012}:

\begin{theorem}\quad\label{Theorem 8}
Let $\mathbb{F}=\mathbb{C}$ or $\mathbb{H}$ and denote by $\hat{U}(2,1:\mathbb{F})$ the isometry group preserving the Hermitian form $\langle \mathbf{z,w}\rangle_4=-(\overline{z_0}w_1+\overline{z_1}w_0)+
\overline{z_2}w_2$ which gives the Siegel domain model. Suppose that $A\in \hat{U}(2,1:\mathbb{F})$
is a loxodromic element. Then $A$ is conjugate to an element of the form
$$L=L(\beta,\theta)=\left(\begin{array}{ccc}
                          re^{i\beta} & 0 & 0 \\
                           0 & r^{-1}e^{i\beta} & 0 \\
                           0 & 0 & e^{i\theta}
                         \end{array}\right)
,\quad r>0,\; r\neq1,
$$
such that $0\leq\beta$, $\theta\leq\pi$ when $\mathbb{F=H}$, and $\pi\leq\beta$, $\theta\leq\pi$ when $\mathbb{F=C}$.
\end{theorem}

Note that the ball model and the Siegel domain model for hyperbolic space are interchanged by Cayley transformations. By Theorem \ref{Theorem 8}, we obtain that every loxodromic equivalence class with respect to the conjugation action of ${\rm Sp}(2,1)$ or ${\rm U}(2,1)$ can be uniquely determined by $r$, $\beta$ and $\theta$ parameters.
\begin{proof2}~

According to Proposition \ref{Proposition1}, we associate the quintuple $(p^+_i,p^-_i,r,\beta,\theta)$ to each $g_i$, where $p^+_i$ and $p^-_i$ are respectively the attracting and the repelling fixed points of $g_i$. Since $G_0$ is discrete, all fixed points $p^+_1,p^-_1,\ldots,p^+_k,p^-_k$ are distinct.

Then
$$(p^+_1,p^-_1,\ldots,p^+_k,p^-_k;r_1\ldots,r_k;\beta_1,\theta_1\ldots,\beta_k,\theta_k),$$
may be associated to $G_0=\langle g_1,\ldots,g_k\rangle$.

Now, according to Theorem \ref{theorem new1}, the proof is completed.\qed
\end{proof2}
\begin{remark}
If we restrict the Theorem \ref{Corollary new1} to complex space, then $\Theta=[-\pi,\pi]$.
\end{remark}
\begin{acknowledgements}
The author G. Gou wishes to express his thanks to Professor I. D. Platis for his warmly help, and Professor J. R. Parker, Professor E. Falbel and professor W. Cao for viewing the preprint of the paper and their suggestions. Besides, authors would like to thank NSFC (No. 11371126) and NSFC (No. 11701165) for financial support. We also thank the referees and the useful references.
\end{acknowledgements}


\begin{thebibliography}{99}
\baselineskip 11.5pt

\bibitem{Alesker29}S. Alesker: Non-commutative linear algebra and plurisubharmonic functions of quaternionic variables. {Bull. Sci. Math}, 127: 1-35, (2003)%
\bibitem{Apanasov2007}B. N. Apanasov , I. Kim: Cartan angular invariant and deformations of rank 1 symmetric spaces. {Sbornik Math}, 198(2): 147-169, (2007)%
\bibitem{Aslaksen15}H. Aslaksen: Quaternionic determinants, {Math. Intelligencer}, 18(3): 57-65, (1996)
\bibitem{Barn.Moore28}R. W. Barnard and E. H. Moore: General analysis. Part 1: Memoirs of the American Philosophical Society, (1935)
\bibitem{Brehm2}U. Brehm: The shape invariant of triangles and trigonometry in two-point homogeneous spaces. {Geom. Dedicata}, 33: 59-76, (1990)%
\bibitem{Brehm3}U. Brehm , B. Et-Taoui: Congruence criteria for finite subsets of complex projective and complex hyperbolic spaces. {Manuscr. Math}, 96(1): 81-95, (1998)%
\bibitem{Cao}W. S. Cao: Congruence classes of points in quaternionic hyperbolic space. {Goem. Dedicata}, 180: 203-228, (2016)%
\bibitem{Cao2012}W. S. Cao, K. Gongopadhyay: Algebraic characterization of isometries of the complex and the quaternionic hyperbolic planes. {Geom. Dedicata}, 157: 23-39, (2012)%
\bibitem{Cao2}W. S. Cao: The moduli space of points in quaternionic projective space. preprint.
\bibitem{Chen1974}S. S. Chen, L. Greenberg: Hyperbolic Spaces, Contributions to Analysis. New York: Academic Press (1974)%
\bibitem{Cunha4}H. Cunha, N. Gusevskii: The moduli space of quadruples of points in the boundary of complex hyperbolic space. {Transform. Groups}, 15(2): 261-283, (2010)%
\bibitem{Cunha5}H. Cunha, N. Gusevskii: The moduli space of points in the boundary of complex hyperbolic space. {J. Geom. Anal.}, 22: 1-11, (2012)%
\bibitem{Cunha6}H. Cunha, F. Dutenhefner, N. Gusevskii, R. Santos Thebaldi: The moduli space of complex geodesics in the complex hyperbolic plane. {J. Geom. Anal.}, 22: 259-319, (2012)%
\bibitem{HakimJ.2000}J. Hakim, H. Sandler: Standard position for objects in hyperbolic space. {J. Geom.}, 68: 100-113, (2000)%
\bibitem{HakimJ.16}J. Hakim, H. Sandler: The moduli space of n + 1 points in complex hyperbolic n-space. {Geom. Dedic.}, 97: 3-15, (2003)%
\bibitem{Falbel8}E. Falbel, I. D. Platis: The PU(2, 1) confguration space of four points in $S^3$ and the cross-ratio variety. {Math. Ann.}, 340(4): 935-962, (2008)
\bibitem{KGongopadhyay19} K. Gongopadhyay, S. Parsad: Classification of quaternionic hyperbolic isometries. {Conform. Geom. Dyn.} 17, 68-76, (2013)%
\bibitem{N.Jacobson22}N. Jacobson: An application of E. H. Moore's determinant of a Hermitian matrix. {Bull. Amer. Math. Soc.}, 45: 745-748, (1939)%
\bibitem{I.Kim17}I. Kim , P. Pansu: Local rigidity in quaternionic hyperbolic space. {Journal of the European Mathematical Society (JEMS)}, 11(6): 1141-1164, (2009)%
\bibitem{KimI.18}I. Kim, J. R. Parker: Geometry of quaternionic hyperbolic manifolds. {Math. Proc. Camb. Philos. Soc.}, 135: 291-320, (2003)%
\bibitem{E.H.Moore7} E. H. Moore: On the determinant of an hermitian matrix of quaternionic elements. {Bull. Amer. Math. Soc.}, 28: 161-162, (1922)%
\bibitem{Parker2008} J. R. Parker, I. D. Platis: Complex hyperbolic Fenchel-Nielsen coordinates. Topology, 47(2): 101-135, (2008)%
\bibitem{Parker2007} J. R. Parker: Hyperbolic Spaces. The Jyv$\mathrm{\ddot{a}}$skyl$\mathrm{\ddot{a}}$ Notes, (2007)%
\bibitem{Platis2014}I. D. Platis: Cross-ratios and the Ptolemaean inequality in boundaries of symmetric spaces of rank 1. {Geometr. Dedicata.}, 169: 187-208, (2014)%
\bibitem{Zhang1997}F. Z. Zhang: Quaternions and matrices of quaternions, {Linear Algebra Appl.}, 251: 21-57, (1997)%
\end{thebibliography}
\end{document}